\documentclass[11pt,twoside,a4paper]{article}
\usepackage{amsmath,amssymb,amsthm,ifpdf}

\usepackage{graphicx,psfrag}

\newtheorem{theorem}{Theorem}[section]
\newtheorem{lemma}[theorem]{Lemma}
\newtheorem{prop}[theorem]{Proposition}
\newtheorem{cor}[theorem]{Corollary}

\theoremstyle{definition}

\theoremstyle{remark}
\newtheorem{remark}[theorem]{Remark}
\newtheorem{definition}[theorem]{Definition}

\newcommand{\R}{\mathbb{R}}

\newcommand{\cM}{\mathcal{M}}

\newcommand{\de}{\delta}

\newcommand{\om}{\omega}
\newcommand{\Om}{\Omega}

\newcommand{\kap}{\kappa}
\newcommand{\la}{\lambda}

\renewcommand{\phi}{\varphi}

\newcommand{\diam}{\operatorname{diam}}

\newcommand{\CAT}{\operatorname{CAT}}
\newcommand{\PT}{\operatorname{PT}}
\newcommand{\hyp}{\operatorname{H}}
\newcommand{\id}{\operatorname{id}}

\newcommand{\crt}{\operatorname{crt}}

\newcommand{\Con}{\operatorname{Con}}

\newcommand{\skap}{\operatorname{sn}_{\kappa}}

\renewcommand{\d}{\partial}

\newcommand{\sm}{\setminus}
\newcommand{\sub}{\subset}

\newcommand{\ov}{\overline}

\begin{document}

\title{Hyperbolic Spaces and Ptolemy M\"obius Structures}
\author{Renlong Miao \&  Viktor Schroeder}

\maketitle

\begin{abstract}
 We characterize the class of Gromov hyperbolic spaces, whose boundary at infinity
allow canonical M\"obius structures.
\end{abstract}

\section{Introduction}

There is a deep and well studied relation between the geometry of the classical 
hyperbolic space and the M\"obius geometry of
its boundary at infinity.
This relation can be generalized in a nice way to
$\CAT (-1)$ spaces. 

Let 
$X$ be a 
$\CAT (-1)$ space with boundary
$Z = \d_{\infty}X$. For every basepoint
$o\in X$ one can define the Bourdon metric
$\rho_o(x,y)=e^{-(x|y)_o}$ on $Z$, where
$(\ |\ )_o$ is the Gromov product with respect to
$o$, compare \cite{B1}.
For different basepoints 
$o,o'\in X$ the metrics $\rho_o,\rho_{o'}$ are
M\"obius equivalent and thus define a 
M\"obius structure on $Z$.
By \cite{FS1} this M\"obius structure is ptolemaic.

On the other hand, examples show that not every ptolemaic M\"obius structure arises 
as boundary of a 
$\CAT (-1)$ space. In this paper we 
enlarge the class of
$\CAT(-1)$ spaces in a way that this larger class corresponds
exactly to the spaces which have a ptolemaic M\"obius structure
at infinity.

\begin{definition}
 A metric space is called
{\em asymptotically $\PT_{-1}$}, if there exists
some $\de >0$ such that for all quadruples
$x_1,x_2,x_3,x_4 \in X$ we have
$$ e^{\frac{1}{2}(\rho_{1,3} +\rho_{2,4})} \ \le\  
e^{\frac{1}{2}(\rho_{1,2} +\rho_{3,4})} \ +\ 
e^{\frac{1}{2}(\rho_{1,4} +\rho_{2,3})} \ +\ \de\, e^{\frac{1}{2}\rho}, $$
where $\rho_{i,j} =d(x_i,x_j)$ and $\rho = \max_{i,j}\rho_{i,j}$.
\end{definition}

We discuss this curvature condition in more detail later and compare it
in section \ref{subsec:ascat} 
with the asymptotically $\CAT(-1)$ condition, which is formulated in more 
familiar comparison terms. 
It turns out that 
$\CAT(-1)$ are 
asymptotically $\PT_{-1}$ 
and that the relation between these spaces and the
M\"obius geometry of their boundaries can be
expressed in the following two results:

\begin{theorem} \label{thm:1}
Let $X$ be asymptotically $\PT_{-1}$, then $X$ is a boundary continuous
Gromov hyperbolic space. For every basepoint $o\in X$,
$\rho_o(x,y) =e^{-(x|y)_o}$ defines a metric on $\d_{\infty}X$.
For different basepoints these metrics
are M\"obius equivalent and thus define a canonical M\"obius structure 
$\cM$ on $\d_{\infty}X$.
The M\"obius structure $\cM$ is complete and ptolemaic.
\end{theorem}

\begin{theorem} \label{thm:2}
Let $(Z,\cM)$ be a complete and ptolemaic M\"obius space. Then there exists an asymptotically
$\PT_{-1}$ space $X$ such that $\d_{\infty}X$ with its canonical
M\"obius structure is M\"obius equivalent to $(Z,\cM)$.
\end{theorem}

The results can be viewed as a characterization of the class of Gromov hyperbolic spaces,
whose boundary allows a canonical M\"obius structure.

In the proof we use a hyperbolic cone construction due to Bonk and Schramm \cite{BoS}.
which associates to a metric space $(Z,d)$ a cone
$(\Con(Z),\rho)$. We show in Proposition \ref{pr:mcone} that if
$(Z,d)$ is ptolemaic, then the cone is asymptotically $\PT_{-1}$.
This method can also be used obtain a characterization of visual Gromov
hyperbolic spaces
in the spirit 
of the \cite{BoS}. Recall that two metric spaces $X$ and $Y$ are roughly
similar, if there are constants $K, \la  > 0$ and a map $f:X\to Y$ such that
for all $x,y \in X$
$$|\la d_X(x,y)\,-\,d_Y(f(x),f(y) |\ \leq K$$
and in additional $\sup_{y\in Y} d_Y(y,f(X)) \le K$.

A theorem of Bonk and Schramm states that a visual Gromov 
hyperbolic space with doubling boundary is roughly similar to a convex subset
of the real hyperbolic space $\hyp^n$ for some integer $n$.

We have a version without conditions on the boundary:

\begin{theorem} \label{thm:3}
Every visual Gromov hyperbolic space is roughly similar to some
asymptotically $\PT_{-1}$ space.
\end{theorem}

We discuss some open questions in Remark \ref{rem:questions}
The structure of the paper is as follows. In section \ref{sec:pre} we
recall the basic facts about metric M\"obius geometry and boundary continuous
Gromov hyperbolic spaces.
In section
\ref{sec:asptkappa} we introduce the $\PT_{\kap}$ property, discuss asymptotically $\PT_{-1}$ spaces
and prove Theorem \ref{thm:1}. In section
\ref{sec:hypcone} we introduce hyperbolic cones and prove Theorem \ref{thm:2}.
The proof of Theorem \ref{thm:3} is in section
\ref{sec:ap}.

\section{Preliminaries} \label{sec:pre}

\subsection{M\"obius Structures} 

Let
$Z$ 
be a set which contains at least
two points.
An {\em extended metric} on 
$Z$ 
is a map 
$d:Z\times Z \to [0,\infty ]$,
such that
there exists
a set
$\Om(d) \sub Z$ with
cardinality
$\# \Om(d) \in \{ 0,1\}$, 
such that 
$d$
restricted to the set
${Z\setminus \Om(d)}$ 
is a metric 
(taking only values in $[0,\infty)$) and such that
$d(z,\omega)=\infty$ 
for all 
$z\in Z\setminus \Om(d)$, $\om \in \Om(d)$. 
Furthermore $d(\omega,\omega)=0$.\\
If
$\Om(d)$
is not empty,
we call the unique
$\om \in \Om(d)$ simply 
{\it the
point at infinity} of
$(Z,d)$.
We write 
$Z_{\om}$ for the set
$Z\sm \{\om\}$.

The topology considered on $(Z,d)$ is the topology with the 
basis consisting of all open distance balls $B_{r}(z)$
around points in $z\in Z_{\om}$ and the
complements $D^C$ of all closed distance balls $D=\overline{B}_r(z)$. \\

We call an extended metric space 
$(Z,d)$ {\em complete}, if first every
Cauchy sequence in $Z_{\om}$ converges and secondly
if the infinitely remote point $\om$ exists in case that $Z_{\om}$
is unbounded.
For example the real line
$(\R,d)$,
with its standard metric is {\em not} complete (as extended metric space),
while
$(\R\cup\{\infty\},d)$ is complete.

We say that a quadruple 
$(x,y,z,w)\in Z^4$ is 
{\em admissible}, if no entry occurs three or
four times in the quadruple.
We denote with
$Q\sub Z^4$
the set of 
admissible quadruples.
We define the {\em cross ratio triple} as the map
$\crt:\ Q \to \Sigma \subset \R P^2$ 
which maps 
admissible quadruples to points in the real projective plane defined by
$$\crt(x,y,z,w)=(d(x,y)d(z,w): d(x,z)d(y,w) : d(x,w)d(y,z)),$$
here 
$\Sigma$ 
is the subset of points 
$(a:b:c) \in \R P^2$, 
where
all entries 
$a,b,c$ 
are nonnegative or all entries are non positive.

We use the standard conventions for the calculation with 
$\infty$.
If 
$\infty$ occurs once in 
$Q$, say 
$w=\infty$,
then
$\crt(x,y,z,\infty)=(d(x,y):d(x,z):d(y,z))$.
If 
$\infty$ 
occurs twice , say $z=w=\infty$ then
$\crt(x,y,\infty,\infty)=(0:1:1)$.

Similar as for the classical cross ratio there are six possible definitions by
permuting the entries and we choose the above one. 

A map
$f:Z\to Z'$
between two extended metric spaces
is called
{\em M\"obius}, if 
$f$ is injective and for all
admissible quadruples
$(x,y,z,w)$ of
$X$,
$$\crt(f(x),f(y),f(z),f(w))=\crt(x,y,z,w).$$
M\"obius maps are continuous.

Two extended metric spaces
$(Z,d)$ and
$(Z,d')$ are
{\em M\"obius equivalent},
if there exists a bijective
M\"obius map
$f:Z\to Z$.
In this case also
$f^{-1}$ is a M\"obius map and
$f$ is in particular a
homeomorphism.\\
We say that two extended metrics
$d$ and $d'$
on the same set 
$Z$ are
{\em M\"obius equivalent},
if the identity map
$\id:(Z,d)\to (Z,d')$
is a M\"obius map.
M\"obius equivalent metrics define
the same topology on
$Z$.
It is also not difficult to check that
for M\"obius equivalent metrics
$d$ and $d'$ the space $(Z,d)$ 
is complete if and only if
$(Z,d')$ is complete.

The M\"obius equivalence of metrics
of metrics on a given set $Z$ is clearly
an equivalence relation.
A {\em M\"obius structure} 
$\cM$ on $Z$
is an equivalence class 
of extended metrics on
$Z$.

A pair $(Z,\cM)$ of a set $Z$ together with a
M\"obius structure $\cM$ on $Z$ is called a 
{\em M\"obius space}.
A M\"obius structure well defines a topology
on $Z$, thus a M\"obius space is in particular a
topological space.
Since completeness is also a M\"obius invariant
we can speak about
{\em complete} M\"obius structures.

In general two metrics in
$\cM$ can look very different.
However if two metrics have the same remote
point at infinity, then they are homothetic (see \cite{FS2}):

\begin{lemma}\label{lem:homothety}
Let $\cM$ be a M\"obius structure on a set
$X$, and let
$d,d'\in \cM$, such that 
$\om \in X$ is the remote point of
$d$ and of $d'$. Then there exists $\la >0$,
such that
$d'(x,y)=\la d(x,y)$ for all 
$x,y \in X$.
\end{lemma}

An extended metric space $(Z,d)$ is called a {\it Ptolemy space}, if for all 
quadruples of points $\{ x,y,z,w\} \in Z^4$
the {\it Ptolemy inequality} holds
\begin{displaymath}
d(x,y) \, d(z,w) \; \le \; d(x,z) \, d(y,w) \; 
+ \; d(x,w) \, d(y,z)
\end{displaymath}

We can reformulate this condition in 
terms of the cross ratio triple. 
Let 
$\Delta \subset \Sigma$ 
be the set of points 
$(a:b:c)\in \Sigma$, such that
the entries 
$a,b,c$ 
satisfy the triangle inequality. This is obviously well defined.

Then an
extended space is  Ptolemy, if 
$\crt(x,y,z,w) \in \Delta$ 
for all allowed quadruples $Q$.

This description shows that 
the Ptolemy property is M\"obius invariant and thus a property
of the M\"obius structure
$\cM$.

The importance of the Ptolemy property comes from the following fact (see e.g. \cite{FS2}):

\begin{theorem} \label{thm:chPT}
A M\"obius structure
$\cM$ on a set
$Z$ is
Ptolemy, if and only if for all
$\om\in Z$ there
exists 
$d_{\om}\in\cM$ with
$\Om(d_{\om})=\{\om\}$. 
\end{theorem}

The metric $d_{\om}$ can be obtained by metric involution. If
$d$ is a metric on $Z$ then 
$$d_{\om}(z,z')=\frac{d(z,z')}{d(\om,z)d(\om,z')}$$
gives the required metric.

\subsection{Boundary continuous Gromov hyperbolic spaces}

We recall some basic facts from the theory of
Gromov hyperbolic spaces, compare e.g \cite{BS}

A metric space $(X,d)$ is called
{\it Gromov hyperbolic}
if there exists
some $\de >0$ such that for all quadruples
$x_1,x_2,x_3,x_4 \in X$ we have
$$\rho_{1,3}+\rho_{2,4} \ \le \ \max \{\rho_{1,2}+\rho_{3,4}\, , \,\rho_{1,4}+\rho_{2,3}\} \ +\ \de,$$
where $\rho_{i,j} =d(x_i,x_j)$.

For three points $x,y,z \in X$ one defines the {\it Gromov product}
$$(x|y)_z\ =\ \frac{1}{2}\,(\, |zx| \,+\, |zy| \,-\, |xy|\,)\ ,$$
where we write $|xy|$ as a short version of $d(x,y)$.

A sequence $(x_i)$ {\it converges at infinity}, if for some (and hence every) 
basepoint $o\in X$ we have
$\lim_{i,j\to\infty} (x_i|x_j)_o =\infty$.
Two such sequences $(x_i)\, ,(y_i)$are called {\it equivalent},
if $\lim (x_i|y_i)_o = \infty$.
The boundary $\d_{\infty}X$ consist of the equivalence classes of these sequences.

For two points $\zeta,\, \xi \in \d_{\infty}X$ and a base point $o\in X$
one defines
\begin{equation} \label{eq:gpbd}
(\zeta|\xi)_o\ =\ \inf \liminf_{i\to \infty} (x_i|y_i)_o 
\end{equation}
where the infimum is taken over all sequences $(x_i) \in \zeta$ and
$(y_i) \in \xi$. 
In a similar way we also define
$(x|\xi)_o$, where $o,x\in X$ and $\xi\in \d_{\infty}X$.

We remark that the sequence
$(x_i|y_i)_o$ does not necessarily converge, therefore
we need the complicated definition in (\ref{eq:gpbd}).

A Gromov hyperbolic space is called
{\it boundary continuous},
if the Gromov product extends continuously to the boundary in the following way:
if $(x_i),\, (y_i)$ are sequences in $X$ which converge to points
$\ov{x},\, \ov{y}$ in $X$ or $\d_{\infty}X$, then
$(x_i|y_i)_o\to(\ov{x}|\ov{y})_o$ for all base points $o\in X$.
For boundary continuous spaces one can define nicely Busemann functions. If
$\om \in \d_{\infty}X$ and $o\in X$ a base point, then
\begin{equation}\label{eq:bf}
 b_{\om,o}(x)=\lim_{i\to\infty}\,(|xw_i|\,-\,|w_io|\,)
\end{equation}
where $w_i\to \om$
is the Busemannfunction of $\om$ normalized to have the value $0$ at the point $o\in X$.
We have the formula:
\begin{equation} \label{eq:bfformula}
b_{\om,o}(x) =(\om|o)_x-(\om|x)_o
 \end{equation}

We also define form
$\om\in \d_{\infty}X$ a base point $o\in X$ and $x,y,z$ from X or $\d_{\infty}X\sm\{\om\}$
$$(x|y)_{\om,o}= (x|y)_o\,-\,(\om|x)_o\,-\,(\om|y)_o.$$


\section{Asymptotic $\PT_{\kappa}$ spaces} \label{sec:asptkappa}

In subsection \ref{subsec:ptkappa} we define general $\PT_{\kappa}$ spaces.
Then in section \ref{subsec:asptkappa} we introduce the more general notion
of asymptotic $\PT_{\kap}$ spaces and 
compare it in section \ref{subsec:ascat} with the notion of
asymptotically $\CAT(\kap)$ spaces.
Finally in section \ref{subsec:properties}
we prove Theorem \ref{thm:1}.

\subsection{The $PT_{\kappa}$ inequality} \label{subsec:ptkappa}

A metric space $(X,d)$ is called a $\PT_{\kappa}$-space, if for points $x_1,x_2,x_3,x_4 \in X$, we have
\begin{equation} \label{eq:ptkappa}
\skap(\frac{\rho_{1,3}}{2})\, \skap(\frac{\rho_{2,4}}{2}) \ \le\ 
\skap(\frac{\rho_{1,2}}{2})\, \skap(\frac{\rho_{3,4}}{2}) \ +\
\skap(\frac{\rho_{1,4}}{2})\, \skap(\frac{\rho_{2,3}}{2})
\end{equation}
where $\rho_{i,j}=d(x_i,x_j)$ and $\skap$ is the function
\begin{equation*}
  \skap(x)=
     \begin{cases}
     \frac{1}{\sqrt{\kappa}}\sin(\sqrt{\kappa}x) & \text{if} \,\,\,\kappa>0, \\
     x & \text{if} \,\,\,\kappa=0, \\
     \frac{1}{\sqrt{-\kappa}}\sinh(\sqrt{-\kappa}x) & \text{if} \,\,\,\kappa<0.
     \end{cases}
\end{equation*}

In the case that $\kap >0$ we assume in addition that the diameter is bounded by
$\frac{\pi}{\sqrt{\kap}}$.

It is well known that the standard space forms $M_{\kappa}^n$ of constant curvature $\kap$
satisfy the $\PT_{\kap}$ inequality. For the euclidean space this is the classical ptolemaic
inequality and for the other spaces it is proved in \cite {H}.
By comparison we obtain the result also for $\CAT(\kap)$-spaces.

\begin{prop}
Every $\CAT(\kap)$ space satisfies the $\PT_{\kap}$ inequality.
\end{prop}
\begin{proof}
A $\CAT(\kappa)$ spaces, $\kappa \in \mathbb{R}$, can be characterized by a $4$-point condition,
\cite[Proposition 1.11]{BH}. Suppose $x_i \in X$  for $0 \leq i \leq 4$, with $x_0=x_4$, and $x_0=x_4$,  there exist four 
points $\bar{x_i} \in M_{\kappa}^2$ with $\bar{x}_0=\bar{x}_4$ such that
\[
  d(x_i,x_{i-1})= |\bar{x}_i-\bar{x}_{i-1}|, 1 \leq i \leq 4,
\]
\[
  d(x_1,x_3)\leq |\bar{x}_1-\bar{x}_3| \;\;\; \text{and} \;\;\; d(x_2,x_4)\leq |\bar{x}_2-\bar{x}_4|.
\]
Since 
$M_{\kappa}^2$ satisfy the $PT_{\kappa}$ inequality the result follows.
\end{proof}

\begin{remark} \label{rem:questions}
The following questions arises naturally: is a geodesic $\PT_{\kap}$ space 
$\CAT(\kap)$? 
A positive answer would imply a nice four point characterization of $\CAT(\kap)$ spaces.
In the case $\kap =0$ this is not true (see \cite{FLS}), but
the counterexamples are not locally compact and there are partial positive results
in the locally compact case (see e.g. \cite{BuFW}, \cite{MS}).
For $\kap < 0$ the question is completely open.
 
\end{remark}

\subsection{Asymptotic $\PT_{\kappa}$ inequality for $\kappa<0$} \label{subsec:asptkappa}

One obtains the asymptotic $\PT_{\kap}$ property (for $\kap < 0$) by weakening equation the 
$\PT_{\kap}$ inequality and allowing some error term. Instead of equation (\ref{eq:ptkappa})
we require that for some universal $\delta \ge 0$ we have

\begin{multline*} 
\skap(\frac{\rho_{1,3}}{2})\,  \skap(\frac{\rho_{2,4}}{2}) \ \ \ 
\le  \\
\skap(\frac{\rho_{1,2}}{2}) \skap(\frac{\rho_{3,4}}{2}) \ +\ 
\skap(\frac{\rho_{1,4}}{2}) \skap(\frac{\rho_{2,3}}{2}) \ + \ \delta e^{\frac{\sqrt{-\kap}}{2}\rho}
\end{multline*}

It is more convenient to formulate this condition using only exponential functions.
It is easy to check that these conditions are equivalent.

\begin{definition}
A metric space is called asymptotic $PT_{\kappa}$ for some $\kap <0$, if there exists 
some $\delta\geq 0$ such that for all quadruples $x_1,x_2,x_3,x_4 \in X$ we have
\[
  e^{\frac{\sqrt{-\kappa}}{2}(\rho_{1,3}+\rho_{2,4})}\leq e^{\frac{\sqrt{-\kappa}}{2}(\rho_{1,2}+\rho_{3,4})}+e^{\frac{\sqrt{-\kappa}}{2}(\rho_{1,4}+\rho_{2,3})}+\delta e^{\frac{\sqrt{-\kappa}}{2}\rho}
\]
Here $\rho_{i,j} =d(x_i,x_j)$ and $\rho=\max_{i,j}\rho_{i,j}$. 
\end{definition}

\begin{remark}
The asymptotic $\PT_{\kap}$ condition is a strong curvature condition. It implies e.g.
that $X$ does not contain flat strips: if a space contains a flat strip of width $a>0$,
then it contains quadruples with
$\rho_{1,3}=\rho_{2,4}=\sqrt{t^2+a^2}$,
$\rho_{1,2}=\rho_{3,4}=t$ and
$\rho_{2,3}=\rho_{1,4}=a$.
These quadruples do not satisfy the asymptotic $\PT_{\kap}$ inequality
for fixed $\kap <0$, $\delta \ge 0$ and $t\to \infty$.
 \end{remark}

\begin{prop}
Let $0 > \kap' > \kap$. If $X$ is
asymptotic $PT_{\kappa}$, then $X$  is asymptotic $PT_{\kappa'}$.
\end{prop}
\begin{proof}
From the asymptotic $PT_{\kappa}$ inequality, we obtain that
\[
  e^{\frac{\sqrt{-\kappa}}{2}(\rho_{1,3}+\rho_{2,4})}\leq e^{\frac{\sqrt{-\kappa}}{2}(\rho_{1,2}+\rho_{3,4})}+e^{\frac{\sqrt{-\kappa}}{2}(\rho_{1,4}+\rho_{2,3})}+\delta e^{\frac{\sqrt{-\kappa}}{2}\rho}
\]
Here $\rho=\max_{i,j}\rho_{i,j}$ for $i,j=1,2,3,4$. \\
Since we know that for $0\leq x \leq 1$
\[
  (a+b)^x\leq a^x+b^x, a >0, b>0.
\]
Hence
\begin{multline*}
e^{\frac{\sqrt{-\kappa'}}{2}(\rho_{1,3}+\rho_{2,4}}= (e^{\frac{\sqrt{-\kappa}}{2}(\rho_{1,3}+\rho_{2,4})})^{\sqrt{\frac{-\kappa'}{-\kappa}}}\\
\leq (e^{\frac{\sqrt{-\kappa}}{2}(\rho_{1,2}+\rho_{3,4})} + e^{\frac{\sqrt{-\kappa}}{2}(\rho_{1,4}+\rho_{2,3})}+\delta e^{\frac{\sqrt{-\kappa}}{2}\rho})^{\sqrt{\frac{-\kappa'}{-\kappa}}}\\
\leq e^{\frac{\sqrt{-\kappa'}}{2}(\rho_{1,2}+\rho_{3,4})} + e^{\frac{\sqrt{-\kappa'}}{2}(\rho_{1,4}+\rho_{2,3})}+\delta' e^{\frac{\sqrt{-\kappa'}}{2}\rho}
\end{multline*}
It satisfies the asymptotic $PT_{\kappa'}$ inequality.
\end{proof}

By scaling an asymptotic $\PT_{\kap}$ space with the factor $\frac{1}{\sqrt{-\kap}}$ we obtain an
asymptotic $\PT_{-1}$ space.
Therefore we will discuss in the sequel mainly $\PT_{-1}$ spaces.

\subsection{Asymptotically $\CAT(\kap)$ spaces} \label{subsec:ascat}

We relate the asymptotic $\PT(\kap)$ property to some condition which is formulated in 
familiar comparison terms.

\begin{definition} \label{def:ascat}
 Let $\kap <0$. A metric space $(X,d)$ is called {\it asymptotically} $\CAT(\kap)$, if
there exists some $\delta \ge 0$ such that for every quadruple of points
$x_1,x_2,x_3,x_4$ in $X$ there are comparison points $\ov x_1,\ov x_2,\ov x_3,\ov x_4$ in
$M^2_{\kap}$ such that 
\newline
$d(x_1,x_2) =|\ov x_1 \ov x_2|,\ \ \  d(x_2,x_3)=|\ov x_2 \ov x_3|,$
\newline
$d(x_3,x_4)=|\ov x_3 \ov x_4|,\ \ \   d(x_4,x_1)=|\ov x_4 \ov x_1|, $
\newline
$d(x_1,x_3)\le |\ov x_1 \ov x_3|,$
\newline  
$\skap(\frac{d(x_2,x_4)}{2})\, \le\, \skap(\frac{|\ov x_2 \ov x_4 |}{2}) \, +\, \delta.$
\end{definition}

\begin{remark}
This definition makes also sense for $\kap =0$, then the last inequality is just
$d(x_2,x_4) \le |\ov x_2 \ov x_4| + 2\delta$. Then the condition is the rough $\CAT(0)$ condition of
\cite{BuF}. In general, if one replaces the last condition ( also for $\kap < 0$) simply by
the condition $d(x_2,x_4) \le |\ov x_2 \ov x_4| + \delta$, then one obtains , what is called
rough $\CAT(\kap)$ in \cite{BuF}. For $\kap < 0$ this is equivalent to Gromov hyperbolicity
and brings no new information
\end{remark}

\begin{lemma}
 If $(X,d)$ is asymptotically $\CAT(\kap)$, then it is also asymptotically
$\PT_{\kap}$.
\end{lemma}

\begin{proof}
Let $X$ be asymptotically $\CAT(\kap)$ with constant $\delta$.
 Let $x_1,x_2,x_3,x_4 \in X$ be given and let 
$\ov x_1,\ov x_2,\ov x_3,\ov x_4 \in M^2_{\kap}$ be comparison points according to the asymptotic $\CAT(\kap)$
property. Let
$\rho_{i,j} = d(x_i,x_j)$, $\ov {\rho}_{i,j} = |\ov x_i \ov x_j|$ and $\rho= \max \rho_{i,j}$. Then
\begin{align*}
 \skap(\frac{\rho_{1,3}}{2})\,  \skap(\frac{\rho_{2,4}}{2}) \ & -\ \frac{\delta}{\sqrt{-\kap}} e^{\frac{\sqrt{-\kappa}}{2}\rho} \\
& \le \skap(\frac{\rho_{1,3}}{2}) (\skap(\frac{\rho_{2,4}}{2})  \,-\,\delta)  \\
& \le \skap(\frac{\ov{\rho}_{1,3}}{2}) \skap(\frac{\ov{\rho}_{2,4}}{2}) \\
& \le  \skap(\frac{\ov{\rho}_{1,2}}{2}) \skap(\frac{\ov{\rho}_{3,4}}{2}) \ +\ 
\skap(\frac{\ov{\rho}_{1,4}}{2}) \skap(\frac{\ov{\rho}_{2,3}}{2})  \\
& = \skap(\frac{\rho_{1,2}}{2})\, \skap(\frac{\rho_{3,4}}{2}) \ +\
\skap(\frac{\rho_{1,4}}{2})\, \skap(\frac{\rho_{2,3}}{2}).
\end{align*}
Thus $X$ is asymptotically $\PT_{\kap}$ with constant $\frac{\delta}{\sqrt{-\kap}}$.

\end{proof}

\subsection{Properties of asymptotically $\PT_{-1}$ spaces} \label{subsec:properties}

\begin{prop}
An asymptotic $PT_{-1}$ metric space is a Gromov hyperbolic space.
\end{prop}
\begin{proof}
The asymptotic $PT_{-1}$ inequality is
\[
  e^{\frac{1}{2}(\rho_{1,3}+\rho_{2,4})}\leq e^{\frac{1}{2}(\rho_{1,2}+\rho_{3,4})}+e^{\frac{1}{2}(\rho_{1,4}+\rho_{2,3})}+\delta e^{\frac{1}{2}\rho}
\]
Using the triangle inequality, we see
\[
\rho \leq \max\{\rho_{1,2}+\rho_{3,4},\rho_{1,4}+\rho_{2,3}\}
\]
which then implies
\[
  e^{\frac{1}{2}(\rho_{1,3}+\rho_{2,4})}\leq (\delta +1)(e^{\frac{1}{2}(\rho_{1,2}+\rho_{3,4})}+e^{\frac{1}{2}(\rho_{1,4}+\rho_{2,3})})
\]
and hence
\[
 \rho_{1,3}+\rho_{2,4} \leq \max\{\rho_{1,2}+\rho_{3,4},\rho_{1,4}+\rho_{2,3}\} + \delta'.
\]
Thus  $X$ is a Gromov hyperbolic space.

\end{proof}

\begin{lemma}
Let $X$ be an asymptotic $\PT_{-1}$ space. Let $(x_i)$ , $(x'_i)$ and $(y_i)$ be sequences in $X$ satisfying 
\[
  \lim_{i\to \infty}(x_i|x'_i)_o=\infty, \;\;\lim_{i\to \infty}(x_i|y_i)_o=a,\;\;o\in X.
\]
Then
\[
  \lim_{i\to \infty}(x'_i|y_i)_o=a
\]
\end{lemma}
\begin{proof}
From the asymptotic $PT_{-1}$ inequality, we obtain
\begin{eqnarray*}
e^{\frac{1}{2}(|x'_iy_i|+|ox_i|)}-e^{\frac{1}{2}(|oy_i|+|x_ix'_i|)}- \delta e^{\frac{1}{2}\rho_i} &\leq& e^{\frac{1}{2}(|ox'_i|+|x_iy_i|)}\\
 &\leq& e^{\frac{1}{2}(|oy_i|+|x_ix'_i|)}+e^{\frac{1}{2}(|x'_iy_i|+|ox_i|)}+ \delta e^{\frac{1}{2}\rho_i},
\end{eqnarray*}
where $\rho_i=\max\{|ox_i|,|ox'_i|,|oy|,|x_ix'_i|,|x_iy_i|,|x'_iy_i|\}$.

\noindent Dividing both sides by $e^{\frac{1}{2}(|ox_i|+|ox'_i|+|oy_i|)}$, we obtain
\[
  e^{-(x'_i|y_i)_o}-e^{-(x_i|x'_i)_o}-E_i \leq e^{-(x_i|y_i)_o}
 \leq e^{-(x'_i|y_i)_o}+e^{-(x_i|x'_i)_o}+E_i,
\]
where
$E_i=\delta e^\frac{1}{2}(\rho_i -|ox_i|-|ox'_i|-|oy_i|)$.
Note that by triangle inequalities
$$|ox_i|+|ox'_i|+|oy_i|-\rho_i\ \ge \min\{|ox_i|,|ox'_i|,2(x_i|x'_i)_o\},$$
and hence $E_i \to 0$ by our assumptions.
Taking the limit, we obtain
\[
   \lim_{i\to \infty}(x'_i|y_i)_o=\lim_{i\to \infty}(x_i|y_i)_o=a.
\]
\end{proof}

As an immediate consequence we get

\begin{cor}
 An asymptotic $\PT_{-1}$ space is boundary continuous.
\end{cor}

\begin{theorem}\label{rPT}
Let $X$ be an asymptotic $PT_{-1}$ metric space and $o \in X$, then
\[
  \rho_o(x,y)=e^{-(x|y)_o},\,\,\, x,y \in \partial_{\infty}X
\]
is a metric on $\partial_{\infty}X$ which is $\PT_0$.
\end{theorem}
\begin{proof}
First, we show that $\rho_o$ is a metric on $\partial_{\infty}X$. 
For given three points $x,y,z \in \partial_{\infty}X$, choose sequences $(x_i) \in x, (y_i)\in y, (z_i) \in z$.
By boundary continuity we have $(x|z)_o=\lim_{i \to \infty} (x_i|z_i)_o$. Then
\[
  e^{-(x|z)_o}=\lim_{i \to \infty}  e^{\frac{1}{2}(|x_iz_i|-|x_io|-|z_io|)}=\lim_{i\to \infty} e^{-\frac{1}{2}(|x_io|+|y_io|+|z_io|)}e^{\frac{1}{2}(|x_iz_i|+|oy_i|)}
\]
From the asymptotic $PT_{-1}$ inequality, we have
\begin{eqnarray*}
e^{\frac{1}{2}(|x_iz_i|+|oy_i|)}&\leq& e^{\frac{1}{2}(|y_iz_i|+|ox_i|)}+e^{\frac{1}{2}(|x_iy_i|+|oz_i|)}+\delta e^{\frac{1}{2}\rho_i}
\end{eqnarray*}
where $\rho_i=\max\{|ox_i|,|oy_i|,|oz_i|,|x_iy_i|,|x_iz_i|,|y_iz_i|\}$.
Thus
\[
  e^{-(x|z)_o}\leq \lim_{i \to \infty} e^{\frac{1}{2}(|x_iy_i|-|ox_i|-|oy_i|)}
+\lim_{i \to \infty} e^{\frac{1}{2}(|y_iz_i|-|oy_i|-|oz_i|)} + \lim_{i\to\infty} E_i,
\]
where $E_i=\delta e^{\frac{1}{2}(\rho_i-|ox_i|-|oy_i|-|oz_i|)}$. Again we easily check that $E_i\to 0$ and we obtain in the limit
the triangle inequality for $\rho_o$.

We use the similar argument to show that $\rho_o$ satisfies the ptolemaic inequality i.e.
\[
  e^{-(x|z)_o-(y|w)_o}\leq e^{-(x|y)_o-(z|w)_o}+e^{-(y|z)_o-(x|w)_o}.
\]
Choose sequences $(x_i)\in x, (y_i) \in y, (z_i)\in z, (w_i)\in w$.
Since we have
\begin{eqnarray*}
  e^{-(x_i|z_i)_o-(y_i|w_i)_o}&=&e^{-\frac{1}{2}(|x_io|+|y_io|+|z_io|+|w_io|)}e^{\frac{1}{2}(|x_iz_i|+|y_iw_i|)}\\
  &\leq& e^{-\frac{1}{2}(|x_io|+|y_io|+|z_io|+|w_io|)}(e^{\frac{1}{2}(|x_iy_i|+|z_iw_i|)}\\
  &&+e^{\frac{1}{2}(|y_iz_i|+|x_iw_i|)} +\delta e^{\frac{1}{2}\rho_i}) \\
  & = & e^{-(x_i|y_i)_o-(z_i|w_i)_o} + e^{-(y_i|z_i)_o-(x_i|w_i)_o} +E_i,
\end{eqnarray*}
where $\rho_i=\max\{|x_iy_i|,|x_xz_i|,|x_iw_i|,|y_iz_i|,|y_iw_i|,|z_iw_i|\}$ and
$$E_i=\delta e^{\frac{1}{2}(\rho_i-|x_io|-|y_io|-|z_io|-|w_io|)}.$$
Again we see that $E_i\to 0$ and we obtain in the limit the desired ptolemaic inequality.

\end{proof}

\begin{remark}
 The above result implies in particular that the asymptotic upper curvature
bound (see \cite{BF}) of an asymptotic $\PT_{\kappa}$ space is bounded above by
$\kap$.
\end{remark}


\section{Hyperbolic cones over M\"obius spaces} \label{sec:hypcone}

In this chapter we prove Theorem \ref{thm:2}.
Therefore we give (bases on \cite{BoS}) a construction, how to associate to a ptolemaic
M\"obius space $(Z,\cM)$ a hyperbolic space $X$ (which turns out to be asymptotically
$\PT_{-1}$), such that $(Z,\cM)$ is the canonical M\"obius structure of $\d_{\infty}X$.

Let $(Z,\cM)$ be a complete ptolemaic M\"obius space.
We choose a point $\om \in Z$ and an extended metric $d \in \cM$ from the M\"obius structure,
such that $\{\om\}=\Om(d)$ is the point at infinity. Such a metric exists by Theorem \ref{thm:chPT} and this
metric is unique (up to homothety) by Lemma \ref{lem:homothety}.

We take now the metric space $(Z_{\om},d)$, where $Z_{\om}=Z\sm \{\om\}$
and apply the cone construction of \cite{BoS} to it.
The space $\Con(Z_{\om})$ has properties analogous to the 
hyperbolic convex hull of a set in the boundary of a real hyperbolic space.
Set
\[
  \Con(Z_{\om})=Z_{\om}\times (0,\infty).
\]

Define $\rho: \Con(Z_{\om}) \times \Con(Z_{\om}) \to [0.\infty)$ by
\begin{equation} \label{eq:mcone}
  \rho((z,h),(z',h'))=2\log(\frac{d(z,z')+h \vee h'}{\sqrt{hh'}}).
\end{equation}

It turns out that $\rho$ satisfies the triangle inequality and is thus a metric, see \cite{BoS}.
We write $|zz'|=d(z,z')$ for $z,z'\in Z_{\om}$.

\begin{prop} \label{pr:mcone}
 $(\Con(Z_{\om}),\rho)$ is asymptotically $PT_{-1}$.
\end{prop}
\begin{proof}
Given arbitrary four points $x_i=(z_i,h_i) \in \Con((Z_{\om},d)), z_i \in (Z_{\om},d)$, $i=1,2,3,4$, we have
\[
  e^{\frac{\rho(x_i,x_j)}{2}}=\frac{|z_iz_j|+h_i \vee h_j}{\sqrt{h_ih_j}},\;\; i\neq j.
\]
i.e.
\begin{equation} \label{eq:ib}
  |z_iz_j|=\sqrt{h_ih_j} e^{\frac{\rho(x_i,x_j)}{2}}-h_i \vee h_j,\;\; i\neq j.
\end{equation}
Since $(Z,\cM)$ is a complete ptolemaic M\"obius space, $(Z_{\om},d)$ is a complete metric space which satisfies the $PT_0$ inequality, 
hence we obtain
\[
  |z_1z_2|\,|z_3z_4|\ +\ |z_1z_4|\,|z_2z_3|\  \geq\  |z_1z_3|\,|z_2z_4|.
\]
Replacing  $|z_iz_j|$ by (\ref{eq:ib}), we have the following inequality
\begin{multline*}
(\sqrt{h_1h_2} e^{\frac{\rho(x_1,x_2)}{2}}-h_1 \vee h_2)(\sqrt{h_3h_4} e^{\frac{\rho(x_3,x_4)}{2}}-h_3 \vee h_4)\\
+(\sqrt{h_1h_4} e^{\frac{\rho(x_1,x_4)}{2}}-h_1 \vee h_4)(\sqrt{h_2h_3} e^{\frac{\rho(x_2,x_3)}{2}}-h_2 \vee h_3)\\
\geq (\sqrt{h_1h_3} e^{\frac{\rho(x_1,x_3)}{2}}-h_1 \vee h_3)(\sqrt{h_2h_4} e^{\frac{\rho(x_2,x_4)}{2}}-h_2 \vee h_4).
\end{multline*}
This can be written as
\begin{multline*}
\sqrt{h_1h_2h_3h_4}(e^{\frac{\rho(x_1,x_2)}{2}+\frac{\rho(x_3,x_4)}{2}}+e^{\frac{\rho(x_1,x_4)}{2}+\frac{\rho(x_2,x_3)}{2}}-e^{\frac{\rho(x_1,x_3)}{2}+\frac{\rho(x_2,x_4)}{2}})\\
-\sqrt{h_1h_2}(h_3 \vee h_4)e^{\frac{\rho(x_1,x_2)}{2}}-\sqrt{h_3h_4}(h_1 \vee h_2)e^{\frac{\rho(x_3,x_4)}{2}}-\sqrt{h_1h_4}(h_2 \vee h_3)e^{\frac{\rho(x_1,x_4)}{2}}\\
-\sqrt{h_2h_3}(h_1 \vee h_4)e^{\frac{\rho(x_2,x_3)}{2}}+\sqrt{h_1h_3}(h_2 \vee h_4)e^{\frac{\rho(x_1,x_3)}{2}}+\sqrt{h_2h_4}(h_1 \vee h_3)e^{\frac{\rho(x_2,x_4)}{2}}\\
+(h_1 \vee h_2)(h_3 \vee h_4)+(h_1 \vee h_4)(h_2 \vee h_3)-(h_1 \vee h_3)(h_2 \vee h_4) \geq 0.
\end{multline*}
Using again (\ref{eq:ib}) we obtain
\begin{multline} \label{eq:ib1}
e^{\frac{\rho(x_1,x_2)}{2}+\frac{\rho(x_3,x_4)}{2}}+e^{\frac{\rho(x_1,x_4)}{2}+\frac{\rho(x_2,x_3)}{2}}-e^{\frac{\rho(x_1,x_3)}{2}+\frac{\rho(x_2,x_4)}{2}}\\
\geq \frac{(h_3 \vee h_4)|z_1z_2|+(h_1 \vee h_2)|z_3z_4|+(h_2 \vee h_3)|z_1z_4|+(h_1 \vee h_4)|z_2z_3|}{\sqrt{h_1h_2h_3h_4}}\\
-\frac{(h_2 \vee h_4)|z_1z_3|+(h_1 \vee h_3)|z_2z_4|}{\sqrt{h_1h_2h_3h_4}}\\
+\frac{(h_1 \vee h_2)(h_3 \vee h_4)+(h_1 \vee h_4)(h_2 \vee h_3)-(h_1 \vee h_3)(h_2 \vee h_4)}{\sqrt{h_1h_2h_3h_4}}
\end{multline}
Since $(a \vee b)(c \vee d)=ac \vee ad \vee bc \vee bd, a,b,c,d \in \mathbb{R}$, we easily obtain that
\[
 (h_1 \vee h_2)(h_3 \vee h_4)+(h_1 \vee h_4)(h_2 \vee h_3) \geq (h_1 \vee h_3)(h_2 \vee h_4)
\]
which shows that the last term in (\ref{eq:ib1}) is nonnegative and can be omitted.
We use below that
\[
  (h_i\vee h_j)+\sqrt{h_ih_j}\geq h_i+h_j
\]
for all $h_i,h_j \geq 0$.\\
Let $\rho=\max_{i,j}\rho_{i,j}$. Then again by (\ref{eq:ib}) $|z_iz_j|\le\sqrt{h_ih_j}e^{\frac{1}{2}\rho}$
and thus
\begin{multline*}
e^{\frac{\rho(x_1,x_2)}{2}+\frac{\rho(x_3,x_4)}{2}}+e^{\frac{\rho(x_1,x_4)}{2}+\frac{\rho(x_2,x_3)}{2}}-e^{\frac{\rho(x_1,x_3)}{2}+\frac{\rho(x_2,x_4)}{2}}+4e^{\frac{1}{2}\rho}\\
\geq \frac{(h_3 \vee h_4)|z_1z_2|+(h_1 \vee h_2)|z_3z_4|+(h_2 \vee h_3)|z_1z_4|+(h_1 \vee h_4)|z_2z_3|}{\sqrt{h_1h_2h_3h_4}}\\
-\frac{(h_2 \vee h_4)|z_1z_3|+(h_1 \vee h_3)|z_2z_4|}{\sqrt{h_1h_2h_3h_4}}+\frac{|z_1z_2|}{\sqrt{h_1h_2}}+\frac{|z_3z_4|}{\sqrt{h_3h_4}}+\frac{|z_1z_4|}{\sqrt{h_1h_4}}+\frac{|z_2z_3|}{\sqrt{h_2h_3}}\\
\geq \frac{(h_3 + h_4)|z_1z_2|+(h_1 + h_2)|z_3z_4|+(h_2 + h_3)|z_1z_4|+(h_1 + h_4)|z_2z_3|}{\sqrt{h_1h_2h_3h_4}}\\
-\frac{(h_2 \vee h_4)|z_1z_3|+(h_1 \vee h_3)|z_2z_4|}{\sqrt{h_1h_2h_3h_4}}\geq 0
\end{multline*}

Therefore $X$ is asymptotic $PT_{-1}$.
\end{proof}

To finish the proof of Theorem \ref{thm:2}, we have to show that $\d_{\infty}X$ can be
canonically identified with $Z$.

We chose a base point $z_0 \in Z_{\om}$ and then $o:=(z_0,1)$ as base point of $X$.
We define for simplicity $|z|:=|zz_0|$.
For $x=(z,h)$ and $x'=(z',h')$ in $X$ we compute 
\begin{equation} \label{eq:Gp}
(x|x')_o\ =\ \log(\frac{(|z|+h\vee 1)(|z'|+ h'\vee 1)}{|zz'|+h\vee h'}).
 \end{equation}

\begin{lemma}
A sequence $x_i=(z_i,h_i)$ in $X$ converges at infinity, if and only if one of the following holds

1. $(z_i)$ is a Cauchy sequence in $Z_{\om}$ and $h_i\to0$.

2. $(|z_i|+h_i)\ \to\ \infty$.
 
\end{lemma}

\begin{proof}
We show first the {\it if} implication:

Assume 1. that $(z_i)$ is a Cauchy sequence and $h_i\to 0$. Then equation (\ref{eq:Gp}) immediately implies
that $\lim_{i,j\to\infty}(z_i|z_j)_o =\infty$.

Assume 2. that $(|z_i|+h_i)\to \infty$. For given $i,j$ let
$$M_{i,j} =\max\{(|z_i|+h_i\vee 1),(|z_j|+h_j\vee 1)\},$$
$$m_{i,j} =\min\{(|z_i|+h_i\vee 1),(|z_j|+h_j\vee 1)\}.$$
One easily sees
$$M_{i,j}\,\ge\,\frac{1}{4}(|z_izj|+h_i\vee h_j )$$
thus
$$(x_i|x_j)_o\ =\ \log(\frac{m_{i,j}M_{i,j}}{|z_iz_j|+h_i\vee h_j})\ \ge\ \log(\frac{1}{4}m_{i,j})$$
and hence $\lim_{i,j\to\infty}(x_i|x_j)_o=\infty$.

For the {\it only if} part assume that we have given a sequence
$x_i=(z_i,h_i)$ with $\lim_{i,j\to\infty}(x_i|x_j)_o=\infty$.

We first show that there cannot 
exist two subsequences $(x_{i_k} )$ and $(x_{i_l} )$ of $(x_i)$, such that $|z_{i_k}|+h_{i_k} \to \infty$ for $k \to \infty$ and $|z_{i_l}|+h_{i_l}\leq M$ 
for all $l$. If to the contrary such sequences would exist, then we easily obtain using triangle inequalities that
\[
  |z_{i_k}|+h_{i_k}\vee 1 -2M -1 \leq |z_{i_k}z_{i_l}|+h_{i_k}\vee h_{i_l} \leq |z_{i_k}|+h_{i_k}\vee 1 +2M+1
\]
and hence $\limsup (x_{i_k}|x_{i_l})_o$ is finite, a contradiction.

Thus either $|z_i|+h_i \to \infty$ and we are in case 2 or $|z_i|+h_i$ is bounded.
The boundedness and $(x_i|x_j)_o\to, \infty$ implies $\log(|z_iz_j|+h_i\vee h_j )\to \infty$ and hence
$(z_i)$ is a Cauchy sequence and $h_i\to 0$.
\end{proof}

\begin{lemma}
 One can identify
$Z$ with $\d_{\infty}X$ in a canonical way.
\end{lemma}

\begin{proof}
 We define a map
$\chi:Z\to \d_{\infty}X$ by
$z\mapsto [(z,\frac{1}{i})]$ for $z\in Z_{\om}$ and
$\om\mapsto [(z_0,i)]$; here $[\ ]$ denotes the equivalence class
of the corresponding sequences. Formula (\ref{eq:Gp}) shows that this map is injective.
Let now $\xi \in \d_{\infty}X$ be given and be represented by
a sequence $x_i=(z_i,h_i)$.
If $|z_i|+h_i \to \infty$ then
$(x_i|(z_0,i))_o \to \infty$ and $\xi=\chi(\om)$.
If $h_i \to 0$ and $(z_i)$ a Cauchy sequence in $Z_{\om}$, then
the $z:=\lim z_i$ exists, since $(Z,\cM)$ is a complete M\"obius structure.
One easily checks $\xi=\chi(z)$.
\end{proof}

\begin{lemma}
 The canonical M\"obius structure of $\d_{\infty}X$ equals to $\cM$.
\end{lemma}

\begin{proof}
 We consider on $\d_{\infty}X$ the canonical M\"obius structure which is given by the metric $\rho_o(z,z')=e^{-(z|z')_o}$.
Using metric involution we consider the extended metric in the same M\"obius class with $\om$ as
infinitely remote point.
This metric is given for $z,z'\in Z_{\om}$ by
$$\rho_{\om,o}(z,z')\ =\ \frac{\rho_o(z,z')}{\rho_o(\om,z)\rho_o(\om,z')}\ =\ e^{-(z|z')_{\om,o}}.$$
Now
$$(z|z')_{\om,o}= (z|z')_o\,-\,(\om|z)_o\,-\,(\om|z')_o.$$
By formula (\ref{eq:Gp}) we have
$$(\om|z)_o=\lim_{i\to\infty}\log(\frac{i(|z_i|+1)}{|z_i|+i}) \ =\ \log(|z|+1)$$
and in the same way $(\om|z')_o=\log(|z'|+1)$.
Using formula (\ref{eq:Gp}) we see that for $z,z'\in Z_{\om}$
$$(z|z')_o\ =\ \log(\frac{(|z|+1)(|z'|+1)}{|zz'|}.$$
Now we easily compute
$$(z|z')_{\om,o}\ =\ -\log(|zz'|),$$
and hence 
$$\rho_{\om,o}(z,z')\ =\ |zz'|.$$

\end{proof}


\section{Proof of Theorem \ref{thm:3}} \label{sec:ap}

Our proof relies on the following result, which is a combination of
Proposition \ref{pr:mcone} and Theorem \cite[Theorem 8.2]{BoS}.
For the notion of a visual Gromov hyperbolic space we also refer to that paper or \cite{BS}.
Two space $(X,d_X)$ and $(Y,d_Y)$ are rough isometric, if there exists $f:X\to Y$ 
and a constant $K \ge 0$ such that
for all $x,y \in X$
$$|d_X(x,y)\,-\,d_Y(f(x),f(y) |\ \leq K$$
and in additional $\sup_{y\in Y} d_Y(y,f(X)) \le K$.

\begin{prop} \label{pr:bl}
 Assume that $X$ is a visual Gromov hyperbolic space such that
$e^{-(\cdot |\cdot)_o}$ is bilipschitz to a ptolemaic metric $d$ on $\d_{\infty}X$,
the $X$ is rough isometric to an asymptotically $\PT_{-1}$ space.
\end{prop}

\begin{proof}
 Consider the truncated Cone
$\Con^T(\d_{\infty}X,d)$, which is defined as
$\Con^T(\d_{\infty}X) =\d_{\infty}X\times(0,D]$,
where $D=\diam(\d_{\infty}X,d)$ again with the metric defined by
(\ref{eq:mcone}). This is the cone considered in \cite{BoS}, where it is shown that
$X$ is rough isometric to $\Con^T(\d_{\infty}X,d)$. Since by Proposition \ref{pr:mcone}
the cone is $\PT_{-1}$, the result follows.
\end{proof}

Now we can finish the proof of Theorem \ref{thm:3}.
We start with some visual Gromov hyperbolic space
$(X,d_X)$ with some base point $o\in X$. 
There exists some $\varepsilon >0$, such that the function
$e^{-\varepsilon (\cdot|\cdot )_o}$ 
is bilipschitz to a metric $\rho(\cdot,\cdot)$ on $\d_{\infty}X$
(see e.g. \cite[Theorem 2.2.7]{BS}).
By a result of Lytchak (see \cite[Proposition 8]{FS1})
$\rho^{\frac{1}{2}}$  is a ptolemaic. 
Clearly $e^{-\frac{\varepsilon}{2} (\cdot|\cdot )_o}$ is bilipschitz to
the metric
$\rho^{\frac{1}{2}}(\cdot,\cdot)$.
Thus the visual Gromov hyperbolic space
$(X,\frac{\varepsilon}{2}d_X)$ satisfies the assumptions of Proposition \ref{pr:bl}
and is rough isometric to an asymptotically $\PT_{-1}$ space.
Hence $(X,d_X)$ is rough similar to this space.


\bigskip

\begin{tabbing}

Renlong Miao,\hskip10em\relax \= Viktor Schroeder,\\ 

Institut f\"ur Mathematik,\>
Institut f\"ur Mathematik, \\

Universit\"at Z\"urich,\> Universit\"at Z\"urich,\\
Winterthurer Strasse 190, \>
 Winterthurer Strasse 190, \\

CH-8057 Z\"urich, Switzerland\>  CH-8057 Z\"urich, Switzerland\\

{\tt mrenlong1988@gmail.com}\> {\tt vschroed@math.uzh.ch}\\

\end{tabbing}

\end{document}